\documentclass[12pt, a4paper]{article}
\usepackage{float}
\usepackage{titlesec} 
\usepackage{geometry}
\usepackage{multirow}
\usepackage{enumerate}
\usepackage{booktabs}
\usepackage{ytableau}
\usepackage{tikz}
\usepackage{graphicx}
\usepackage{caption}
\usepackage{subcaption}
\usepackage{mathrsfs}
\usepackage{makecell}
\usepackage{algorithm,algorithmic}
\usepackage{lmodern}
\usepackage{amssymb}
\usepackage{amsfonts}
\usepackage{epstopdf}
\usepackage{setspace}
\usepackage{appendix}
\usepackage{amsmath}
\usepackage{amsthm}
\usepackage[
    bookmarks=false,
    unicode=true,
    colorlinks=true,
    linkcolor=blue,
    anchorcolor=blue,
    citecolor=green
]{hyperref}

\usepackage[running]{lineno}

\def\fix{\mathrm{fix}}

\usepackage{authblk}

\usepackage{xcolor}
\usepackage[normalem]{ulem}

\allowdisplaybreaks[4]

\newtheorem{thm}{Theorem}[section]
\newtheorem{lem}[thm]{Lemma}

\newtheorem{cor}[thm]{Corollary}
\newtheorem{prob}[thm]{Problem}

\theoremstyle{definition}
\newtheorem{def1}{Definition}[section]

\title{The Aldous property for normal Cayley graphs on symmetric groups}

\author[1]{Chenhui Lv\thanks{Corresponding author.}}
\author[2]{Sanming Zhou}
\affil[1]{\small School of Mathematical Sciences, University of Science and Technology of China, Hefei 230026, People's Republic of China}
\affil[2]{\small School of Mathematics and Statistics, The University of Melbourne, Parkville, VIC 3010, Australia}

\begin{document}

\maketitle
\openup 0.6 \jot 

\renewcommand{\thefootnote}{\empty}
\footnotetext{E-mail addresses: lch1994@mail.ustc.edu.cn (Chenhui Lv),  sanming@unimelb.edu.au (Sanming Zhou)}

\begin{abstract}
Aldous' spectral gap conjecture states that the random walk and the interchange process on any connected graph have the same spectral gap, or, equivalently, the second largest eigenvalue of any connected Cayley graph on the symmetric group $S_n$ with respect to a set of transpositions is achieved by the standard representation of $S_n$. This celebrated conjecture, proved in its general form in 2010, has inspired much interest in searching for other Cayley graphs on $S_n$ possessing this property, now known as the Aldous property. In this paper, we first prove that for $n \ge 5$ at most one of a normal Cayley graph on $S_n$ and its complement can possess the Aldous property except when these two graphs are $2K_{n!/2}$ and $K_{n!/2,n!/2}$ respectively. We then determine, for sufficiently large $n$, all normal Cayley graphs $\mathrm{Cay}(S_n, S)$ that have the Aldous property, except for the case when $S$ contains a permutation with support size in $\{2, 3, \dots, n-2\}$ and a permutation with support size in $\{n-1, n\}$, but not all permutations with support size $n$ are contained in $S$. In particular, we show that a non-complete normal Cayley graph $\mathrm{Cay}(S_n, S)$ does not have the Aldous property if all permutations in $S$ have support size $n-1$ or $n$, or all permutations with support size $n$ are contained in $S$, thereby solving an open problem posed by Li, Xia and Zhou in 2023. Along the way we determine all normal Cayley graphs on $S_n$ that are line graphs, and classify all normal Cayley graphs on $S_n$ with the strictly second largest eigenvalue at most $1$.
\end{abstract}

\vspace{1em}

\vspace{1em}

\section{Introduction}\label{sec:intro}

All graphs considered in this paper are finite and undirected with no loops or parallel edges. The order of a graph is meant to be its number of vertices. Denote by $A(\Gamma)$ the adjacency matrix of a graph $\Gamma$, and call its eigenvalues the \emph{eigenvalues} of $\Gamma$. Since $A(\Gamma)$ is a real symmetric matrix, the eigenvalues of $\Gamma$ are real numbers, and we order them as $\lambda_1(\Gamma) \geq \lambda_2(\Gamma) \geq \dots \geq \lambda_n(\Gamma)$, where $n$ is the order of $\Gamma$. Similarly, for a real symmetric matrix $M$, we write $\lambda_i(M)$ for its $i$-th largest eigenvalue. The \emph{spectrum} of $\Gamma$, denoted by $\mathrm{Spec}(\Gamma)$, is the multiset of eigenvalues of $\Gamma$. The difference $\lambda_1(\Gamma) - \lambda_2(\Gamma)$ is called the \emph{spectral gap} of $\Gamma$. If  $\lambda_1(\Gamma)$ has multiplicity $t$, then we call $\lambda_{t+1}(\Gamma)$ the \emph{strictly second largest eigenvalue} of $\Gamma$. In the case when $\Gamma$ is $d$-regular, we have $\lambda_1(\Gamma) = d$ with multiplicity the number of components of $\Gamma$, and the spectral gap $d - \lambda_2(\Gamma)$ equals the second smallest Laplacian eigenvalue of $\Gamma$.

Let $G$ be a finite group with identity element $e$, and let $S$ be an inverse-closed subset of $G \setminus \{e\}$. The \emph{Cayley graph} on $G$ with respect to the \emph{connection set} $S$, denoted by $\mathrm{Cay}(G, S)$, is the $|S|$-regular graph with vertex set $G$ and edge set $\{\{g, gs\}: g \in G, s \in S\}$. It is readily seen that $\mathrm{Cay}(G, S)$ has $t = [G:\langle S \rangle]$ connected components, so its strictly second largest eigenvalue is $\lambda_{t+1}(\mathrm{Cay}(G, S))$, where $[G:\langle S \rangle]$ is the index in $G$ of the subgroup $\langle S \rangle$ generated by $S$. We call $\mathrm{Cay}(G, S)$ \emph{normal} if $S$ is closed under conjugation (that is, $S$ is the union of some conjugacy classes of $G$) and \emph{nonnormal} otherwise.

\subsection{Aldous' spectral gap conjecture}

Given a graph $\Gamma$ on $[n] = \{1, 2, \ldots, n\}$, we may identify each edge $\{i,j\}$ of $\Gamma$ with the transposition $(i\ j)$ of the symmetric group $S_n$ on $[n]$. Let $T(\Gamma)$ be the set of such transpositions of $S_n$. Then $\mathrm{Cay}(S_n, T(\Gamma))$ is a Cayley graph on $S_n$ associated with $\Gamma$. It is readily seen that $\mathrm{Cay}(S_n, T(\Gamma))$ is connected whenever $\Gamma$ is connected. Around 1992, Aldous \cite{A1} formulated a stunning conjecture, which is widely known as \emph{Aldous' spectral gap conjecture}. It asserts that any connected graph $\Gamma$ and its associated Cayley graph $\mathrm{Cay}(S_n, T(\Gamma))$ have the same spectral gap. In other words, the random walk and the interchange process on any connected graph have the same spectral gap. When $\Gamma$ is a complete graph \cite{DS} or a star \cite{FOW}, this was proved before Aldous made his conjecture, and when $\Gamma$ is a complete multipartite graph the conjecture was proved in \cite{C}. After a series of partial successes (see \cite{C3} and the references therein), Aldous' spectral gap conjecture was finally proved by Caputo, Liggett and Richthammer \cite{CLR} in 2010. In fact, what they proved is a more general version of Aldous' spectral gap conjecture, namely when each edge of $\Gamma$ has a positive weight which carries forward to corresponding edges of $\mathrm{Cay}(S_n, T(\Gamma))$. See \cite{CLR, C2, C3} and \cite[Section 8.1]{LZ} for details. In this paper we confine ourselves to the unweighted case.

Aldous' spectral gap conjecture can be stated in terms of representations of symmetric groups. Let $G$ be a group, and let $\widehat{G}=\{\rho_1,\rho_2,\ldots,\rho_k\}$ denote a complete set of inequivalent complex irreducible matrix representations of $G$, with the assumption that $\rho_1$ is the trivial representation. For each $\rho_i \in \widehat{G}$, let $\chi_i$ be the character of $\rho_i$, and let $\tilde{\chi}_i(g):=\chi_i(g)/\chi_i(e)$
be the normalized character of $\rho_i$ at $g\in G$, where $\chi_i(e) = \dim\rho_i$ is the dimension of $\rho_i$. It is known \cite{MS} that the adjacency matrix of any Cayley graph $\mathrm{Cay}(G, S)$ on $G$ is equal to $\sum_{s\in S} R_{\mathrm{reg}}(s)$, where $R_{\mathrm{reg}}$ is the right regular representation of $G$ and $R_{\mathrm{reg}}(s)$ is the permutation matrix representing right multiplication by $s$ on $G$. Set
$\rho_i(S) := \sum_{s \in S} \rho_i(s)$,
and denote by $\oplus$ the direct sum of matrices. The adjacency matrix of $\mathrm{Cay}(G,S)$ is similar to
$d_1\rho_1(S)\oplus d_2\rho_2(S)\oplus \cdots\oplus d_k\rho_k(S)$
(see \cite[Proposition 7.1]{MS}), where $d_i=\dim\rho_i$ and $d_i\rho_i(S)$ denotes the direct sum of $d_i$ copies of $\rho_i(S)$.
This implies that the spectrum of $\mathrm{Cay}(G,S)$ is the union of $d_i$ copies of the multiset of eigenvalues of $\rho_i(S)$, for $1\leq i\leq k$. We say that an eigenvalue $\mu$ of $\mathrm{Cay}(G,S)$ is \emph{attained} \cite{LXZ} by the representation $\rho_i$ of $G$ if $\mu=\lambda_1(\rho_i(S))$. In particular, the strictly second largest eigenvalue $\lambda_{t+1}(\mathrm{Cay}(G,S))$ of $\mathrm{Cay}(G,S)$ (where $t = [G:\langle S \rangle]$) is \emph{attained} by $\rho_i$ if
$\lambda_{t+1}(\mathrm{Cay}(G,S))=\lambda_1(\rho_i(S))$.
If $\mathrm{Cay}(G,S)$ is normal, then Schur's lemma implies that each $\rho_i(S)$ is a scalar matrix, and the eigenvalues of $\mathrm{Cay}(G,S)$ are given by (\cite{DS,Z})
\begin{equation}
\label{eq:muj}
\mu_j=\frac{1}{\chi_j(e)}\sum_{s\in S}\chi_j(s)=\sum_{s\in S}\tilde{\chi}_j(s),\quad j=1,2,\ldots,k.
\end{equation}
Moreover, the multiplicity of $\mu_j$ is equal to $\sum_{1\leq i\leq k,\, \mu_i=\mu_j}\chi_i(e)^2$. Thus, if $\mathrm{Cay}(G,S)$ is normal, then an eigenvalue $\mu$ of $\mathrm{Cay}(G,S)$ is attained by $\rho_i$ if and only if $\mu = \sum_{s\in S} \tilde{\chi}_i(s)$, and in particular the strictly second largest eigenvalue of $\mathrm{Cay}(G,S)$ is attained by $\rho_i$ if and only if $\lambda_{t+1}(\mathrm{Cay}(G,S))=\sum_{s\in S} \tilde{\chi}_i(s)$.

It is known \cite{Sagan} that each partition of $n$ gives rise to an irreducible representation of $S_n$, known as a Specht module, and such Specht modules form a complete list $\widehat{S_n}$ of inequivalent irreducible representations of $S_n$. The Specht modules corresponding to the partitions $(n-1,1)$ and $(1^n)$ of $n$ are called the \emph{standard representation} and \emph{sign representation} of $S_n$ and are denoted by $\rho_{(n-1,1)}$ and $\rho_{(1^n)}$, respectively. As explained in \cite{C,C2,PP}, Aldous' spectral gap conjecture is equivalent to saying that for any connected graph $\Gamma$ with vertex set $[n]$ the second largest eigenvalue of $\mathrm{Cay}(S_n, T(\Gamma))$ is attained by $\rho_{(n-1,1)}$; that is, $\lambda_{2}(\mathrm{Cay}(S_n, T(\Gamma))) = \lambda_1(\rho_{(n-1,1)}(T(\Gamma)))$, or $\lambda_{2}(\mathrm{Cay}(S_n, T(\Gamma))) = \sum_{\sigma \in T(\Gamma)} \tilde{\chi}_{(n-1,1)}(\sigma)$ when $\mathrm{Cay}(S_n, T(\Gamma))$ is normal. This version of Aldous' spectral gap conjecture motivated the following notion for generalizing the conjecture (see \cite{AKP,C3,K,PP} for other attempts to generalize Aldous' spectral gap conjecture). 

\begin{def1}
[{\cite[Definition 1.1]{LXZ}}] 
A Cayley graph $\mathrm{Cay}(S_n, S)$ on $S_n$ is said to have the \emph{Aldous property} if its strictly second largest eigenvalue is attained by the standard representation of $S_n$, that is, $\lambda_{t+1}(\mathrm{Cay}(S_n, S)) = \lambda_1(\rho_{(n-1,1)}(S))$, where $t = [S_n:\langle S \rangle]$.  
\end{def1}

In particular, if $\mathrm{Cay}(S_n, S)$ is normal, then it has the Aldous property if and only if $\lambda_{t+1}(\mathrm{Cay}(S_n, S)) = \sum_{\sigma \in S} \tilde{\chi}_{(n-1,1)}(\sigma)$. 

It would be natural to study the following problem. 

\begin{prob}
\label{prob}
Determine all Cayley graphs on $S_n$ that have the Aldous property.
\end{prob}

In this paper we study this problem with a focus on normal Cayley graphs on $S_n$. Before discussing our results, let us give a brief account of recent progress on Problem \ref{prob}. As usual, for a permutation $\sigma\in S_n$, let $\mathrm{supp}(\sigma)=\{i\in [n] : \sigma(i)\ne i\}$ be the \emph{support} of $\sigma$ and $\fix(\sigma) = \{i \in [n] : \sigma(i) = i\}$ the set of \emph{fixed points} of $\sigma$. For $I\subseteq \{2,3,\ldots,n\}$ and $2\leq k\leq n$, set
$$
\mathrm{Supp}_n(I) = \{ \sigma \in S_n : |\mathrm{supp}(\sigma)| \in I \}
$$
and
$$
\mathrm{Supp}_n(k) = \{ \sigma \in S_n : |\mathrm{supp}(\sigma)| = k \}.
$$
For $1 \leq r < k < n$ and $I\subseteq \{2,3,\ldots,n\}$, let $C(n,k)$ be the set of all $k$-cycles in $S_n$, $C(n,k;r)$ the set of all $k$-cycles $\sigma$ in $S_n$ with $\fix(\sigma) \subseteq \{r+1, \ldots, n\}$, and $C(n,I)=\bigcup_{k\in I} C(n,k)$. In \cite{LXZ}, Li, Xia and Zhou proved the following results for sufficiently large $n$ with the help of \cite[Proposition 2.3]{PP}: if $S$ is a single conjugacy class of $S_n$, then $\mathrm{Cay}(S_n,S)$ has the Aldous property if and only if $2\leq |\mathrm{supp}(\sigma)|\leq n-2$ for each $\sigma\in S$; for $\emptyset\ne I\subseteq \{2,3,\ldots,n\}$ with $|I\cap \{n-1,n\}|\ne 1$, $\mathrm{Cay}(S_n,\mathrm{Supp}_n(I))$ has the Aldous property if and only if $I\cap \{n-1,n\}=\emptyset$; for any $2 \le k \le n$, $\mathrm{Cay}(S_n,\mathrm{Supp}_n(I))$ where $I=\{2, \ldots, k\}$ has the Aldous property. In \cite{LXZJCTA2024}, the same authors proved that the strictly second largest eigenvalue of $\mathrm{Cay}(S_n,C(n,I))$ can be attained by at most four distinct irreducible representations of $S_n$, and for some special $I\subseteq \{2,3,\ldots,n\}$ they further determined the exact value of this eigenvalue and its multiplicity. As a corollary, they gave a necessary and sufficient condition for $\mathrm{Cay}(S_n,C(n,I))$ to have the Aldous property when $I\cap \{n-1,n\}=\emptyset$, and showed that this graph does not have the Aldous property whenever $n\in I$. Another corollary of the main results in \cite{LXZJCTA2024} states that for any $2 \le k \le n-2$, $\mathrm{Cay}(S_n,C(n,k))$ has the Aldous property, thereby proving a conjecture in \cite[Conjecture 1.4]{MR} and solving an open problem in \cite{SZ}. (See \cite{DS,HH,HH2,MR} for the proof of this conjecture when $k=2,3,4,5$ respectively.) 

In \cite{SZ}, Siemons and Zalesski conjectured that for $n \ge 5$ and $1 \leq r < k < n$ the nonnormal Cayley graph $\mathrm{Cay}(S_n,C(n,k;r))$ has the Aldous property. They proved in the same paper that this is true when $r=k-1$, and the case when $k = 2,3$ was established earlier in \cite{FOW} and \cite{HH}, respectively. In \cite[Theorem 1.3]{LXZJCTA2026}, Li, Xia and Zhou settled this conjecture by showing that for $n \ge 5$, $k \ge 4$ and $1 \leq r < k < n$, $\mathrm{Cay}(S_n,C(n,k;r))$ has the Aldous property except for (i) $(n,k,r) = (6,5,1)$, or (ii) $n$ is odd, $k=n-1$ and $1 \leq r < \frac{n}{2}$. In \cite{HH2}, Huang, Huang and Cioab\u{a} proved that most connected normal Cayley graphs on $S_n$ with $n\geq 7$ whose connection sets consist of permutations moving at most five points have the Aldous property. 
In the same paper, they posed two conjectures (\cite[Conjectures 23 and 24]{HH2}) regarding the quotient matrices of the full-flag Johnson graph $FJ(n, 2)$. These conjectures have recently been confirmed by Greaves and Zhu \cite{GZ2026}, establishing that $FJ(n, 2)$ possesses the Aldous property.
In \cite{C1}, Cesi showed that the pancake graph $P_n$ has the Aldous property, and Chung and Tobin~\cite{CT} generalized this result to a family of graphs containing all pancake graphs. The reader is referred to \cite[Section 8]{LZ} for more results on the Aldous property and the second largest eigenvalue of Cayley graphs. 

\subsection{Main results} 

A subset $S$ of $S_n \setminus \{e\}$ is called \cite{PP} a \emph{normal set} in $S_n$ if it is the union of some conjugacy classes of $S_n$. Since in $S_n$ every element is conjugate to its inverse element, a Cayley graph $\mathrm{Cay}(S_n,S)$ on $S_n$ is normal if and only if $S$ is a normal set in $S_n$. Clearly, for a normal set $S$ in $S_n$, the set 
$$
S^c=(S_n\setminus\{e\})\setminus S
$$
is also a normal set in $S_n$, and moreover $\mathrm{Cay}(S_n,S)$ and $\mathrm{Cay}(S_n,S^c)$ are complements of each other. Throughout the paper, we denote by $K_n$ the complete graph of order $n$, by $K_{m,n}$ the complete bipartite graph with $m$ and $n$ vertices in the two partite sets respectively, and by $t\Gamma$ the vertex-disjoint union of $t$ copies of a graph $\Gamma$. The first main result in this paper is as follows. 

\begin{thm}\label{thm:complement}
Let $n \ge 5$, and let $S$ be a nonempty normal set in $S_n$ with $S\neq S_n\setminus\{e\}$.
\begin{enumerate}[\rm (a)]
\item
If neither $S$ nor $S^c$ is equal to $A_n\setminus\{e\}$, then at most one of $\mathrm{Cay}(S_n,S)$ and $\mathrm{Cay}(S_n,S^c)$ has the Aldous property.
\item 
If either $S$ or $S^c$ is equal to $A_n\setminus\{e\}$, then one of $\mathrm{Cay}(S_n, S)$ and $\mathrm{Cay}(S_n, S^c)$ is isomorphic to $2K_{n!/2}$ and the other to $K_{n!/2,n!/2}$, and moreover both graphs have the Aldous property. 
\end{enumerate}
\end{thm}

Theorem~\ref{thm:complement} implies that at most half of all normal Cayley graphs on $S_n$ exhibit the Aldous property. This is true no matter whether isomorphism is taken into consideration because two normal Cayley graphs are isomorphic if and only if their complements are isomorphic.

The second main result in this paper, Theorem \ref{thm:main1} below, settles Problem \ref{prob} for sufficiently large $n$ and normal Cayley graphs whose connection sets consist of permutations with support size between $2$ and $n-2$. For a subset $S\subseteq S_n$, denote by $O(S)$ and $E(S)$ the sets of odd and even permutations in $S$, respectively. 
 
\begin{thm}\label{thm:main1}
There exists a positive integer $N \geq 5$ such that, for every $n\ge N$ and every nonempty normal set $S$ in $S_n$ satisfying $2 \le |\mathrm{supp}(\sigma)| \le n-2$ for all  $\sigma\in S$, the following statements hold:
\begin{enumerate}[\rm (a)]
\item 
if $E(S)=\emptyset$ or $O(S)=\emptyset$, then $\mathrm{Cay}(S_n,S)$ has the Aldous property;

\item 
if $E(S)\ne\emptyset$ and $O(S)\ne\emptyset$, then $\mathrm{Cay}(S_n,S)$ has the Aldous property if and only if
\begin{equation}
\label{eq:fix}
\sum_{\sigma\in S} |\fix(\sigma)| \ge n|E(S)|-(n-2)|O(S)|;
\end{equation}
moreover, if $\mathrm{Cay}(S_n,S)$ does not have the Aldous property, then its strictly second largest eigenvalue is attained by the sign representation of $S_n$.
\end{enumerate}
\end{thm}

If $2 \le |\mathrm{supp}(\sigma)| \le n-2$ for all $\sigma\in S$, then $\sum_{\sigma\in S} |\fix(\sigma)| \geq 2|S| = 2|E(S)|+2|O(S)|$, and hence \eqref{eq:fix} is satisfied when $n|O(S)| \ge (n-2)|E(S)|$. So we obtain the following corollary of Theorem~\ref{thm:main1} immediately. 

\begin{cor}\label{odd-even}
There exists a positive integer $N \ge 5$ such that, for every $n \ge N$ and every nonempty normal set $S$ in $S_n$ satisfying $2 \le |\mathrm{supp}(\sigma)| \le n-2$ for all $\sigma \in S$, if $n|O(S)| \ge (n-2)|E(S)|$, then $\mathrm{Cay}(S_n, S)$ has the Aldous property.
\end{cor}

Another corollary of Theorem~\ref{thm:main1} is the following explicit formula for the second largest eigenvalue of normal Cayley graphs on the alternating group $A_n$.

\begin{cor}\label{An-sec-eigen}
There exists a positive integer $N \ge 5$ such that, for every $n \ge N$ and every nonempty normal set $S$ in $S_n$, if $S \subseteq A_n$ and $2 \le |\mathrm{supp}(\sigma)| \le n-2$ for all $\sigma \in S$, then the strictly second largest eigenvalue of $\mathrm{Cay}(S_n, S)$ is given by
$$
\frac{1}{n-1} \big(\sum_{\sigma \in S} |\fix(\sigma)| - |S| \big),
$$
and moreover $\mathrm{Cay}(S_n, S)$ is disconnected with two components each isomorphic to $\mathrm{Cay}(A_n, S)$. 
\end{cor}

For a subset $S \subseteq S_n$, let
\begin{equation}
\label{eq:ABCS}
A_S = S \cap \mathrm{Supp}_n(\{2,3,\ldots,n-2\}),\; B_S = S \cap \mathrm{Supp}_n(n-1),\; C_S = S \cap \mathrm{Supp}_n(n)
\end{equation}
and
\begin{equation}
\label{eq:barABCS}
\overline{A}_S = \mathrm{Supp}_n(\{2,3,\ldots,n-2\}) \setminus A_S,\; \overline{B}_S = \mathrm{Supp}_n(n-1)\setminus B_S,\; \overline{C}_S = \mathrm{Supp}_n(n)\setminus C_S.
\end{equation}
The following theorem classifies all normal Cayley graphs $\mathrm{Cay}(S_n, S)$ that have the Aldous property under an easy-to-check condition.

\begin{thm}\label{thm:main2}
Let $n\ge 5$, and let $S$ be a nonempty normal set in $S_n$. If  
\begin{equation}
\label{eq:ABC}
\sum_{\sigma\in A_S} |\fix(\sigma)| \le |A_S|+|C_S|+n-1,
\end{equation}
then $\mathrm{Cay}(S_n, S)$ has the Aldous property if and only if $S = S_n\setminus\{e\}, S_n\setminus A_n$ or $A_n\setminus\{e\}$, or, equivalently, $\mathrm{Cay}(S_n,S)\cong K_{n!}, K_{n!/2,n!/2}$ or $2K_{n!/2}$, respectively.
\end{thm}

Note that condition \eqref{eq:ABC} holds only if either $A_S$ is empty or $C_S$ is nonempty. Theorem \ref{thm:main2} implies the following corollary.

\begin{cor}\label{cor:main2}
Let $n\ge 5$, and let $S$ be a nonempty normal set in $S_n$. Then $\mathrm{Cay}(S_n,S)$ does not have the Aldous property if one of the following conditions holds:
\begin{enumerate}[\rm (a)]
\item $S \subseteq \mathrm{Supp}_n(\{n-1,n\})$;
\item $|\overline{A}_S|-|\overline{C}_S|+2(n-1) \ge 0$ and $S \notin \{S_n\setminus\{e\}, S_n\setminus A_n, A_n\setminus\{e\}\}$;
\item $\mathrm{Supp}_n(n)\subseteq S$ and $S \neq S_n\setminus\{e\}$.
\end{enumerate}
\end{cor}

In \cite[Problem 1.2]{LXZ}, Li, Xia and Zhou ask for a necessary and sufficient condition for $\mathrm{Cay}(S_n, \mathrm{Supp}_n(I))$ with $\{n\} \subset I \subseteq \{2,3,\ldots,n-2,n\}$ to have the Aldous property for sufficiently large $n$. Solving this problem, part (c) of Corollary \ref{cor:main2} implies that such a graph $\mathrm{Cay}(S_n, \mathrm{Supp}_n(I))$ does not have the Aldous property for any $n \geq 5$. 

Theorem \ref{thm:main1} and Corollary \ref{cor:main2} together resolve Problem \ref{prob} for all normal Cayley graphs $\mathrm{Cay}(S_n, S)$ with sufficiently large $n$ except when $S \cap \mathrm{Supp}_n(\{2, 3, \dots, n-2\}) \neq \emptyset$, $S \cap \mathrm{Supp}_n(\{n-1, n\}) \neq \emptyset$, and $\mathrm{Supp}_n(n) \not\subseteq S$.

The rest of the paper is organized as follows. The next section is a quick review of representation theory preliminaries. We prove Theorem~\ref{thm:complement} in Section~\ref{sec:complements} and establish Theorem~\ref{thm:main1} and Corollary~\ref{An-sec-eigen} in Section~\ref{supp-small}. To prove Theorem \ref{thm:main2}, we first need to determine all normal Cayley graphs on $S_n$ that are line graphs (see Theorem \ref{class-linegraph}). Using this and the classification \cite{Stanic2010} of connected regular graphs with second largest eigenvalue at most $1$, we obtain a classification of all normal Cayley graphs on $S_n$ whose strictly second largest eigenvalue is at most $1$ (see Theorem~\ref{sn-cay-sec-1}). With the help of this classification, which is of interest for its own sake, we prove Theorem~\ref{thm:main2} and Corollary \ref{cor:main2} in Section~\ref{supp-large}.

\section{Preliminaries}\label{sec:preliminaries}

The reader is referred to \cite{I,JL,S1,S} for general representation theory and \cite{Sagan} for symmetric groups and their representations. A \emph{partition} of a positive integer $n$ is a sequence of positive integers $\gamma=(\gamma_1,\gamma_2,\ldots,\gamma_m)$ such that $\gamma_1\ge \gamma_2\ge\cdots\ge\gamma_m$ and $n=\gamma_1+\gamma_2+\cdots+\gamma_m$. We write $\gamma \vdash n$ to indicate this partition of $n$, and let $c_{i}(\gamma)$ denote the number of parts of $\gamma$ which are equal to $i$. The \emph{sign} of a permutation $\sigma\in S_n$, denoted by $\mathrm{sgn}(\sigma)$, is defined to be $1$ if $\sigma$ is an even permutation and $-1$ if $\sigma$ is an odd permutation. Every permutation $\sigma\in S_n$ decomposes into disjoint cycles; the \emph{cycle type} of $\sigma$ is the partition of $n$ whose parts are the lengths of these cycles. It is well known that two elements of $S_n$ are conjugate if and only if they have the same cycle type. Thus, the conjugacy classes of $S_n$ are determined by their cycle types and thus correspond to the partitions of $n$. For a partition $\gamma\vdash n$, denote by $C(S_n,\gamma)$ the corresponding conjugacy class of $S_n$. Note that $c_{1}(\gamma) = |\fix(\sigma)|$ is the number of fixed points of each $\sigma \in C(S_n,\gamma)$. Define $\mathrm{sgn}(\gamma)=1$ if the permutations in $C(S_n,\gamma)$ are even, and $\mathrm{sgn}(\gamma)=-1$ otherwise. 

For each partition $\zeta\vdash n$, let $\rho_\zeta$ denote the corresponding Specht module in $\widehat{S_n}$. The Specht modules $\rho_{(n)}$, $\rho_{(1^n)}$ and $\rho_{(n-1,1)}$ are called the \emph{trivial representation}, \emph{sign representation} and \emph{standard representation} of $S_n$, respectively. We denote the character and normalized character of $\rho_\zeta$ by $\chi_\zeta(\cdot)$ and $\tilde{\chi}_\zeta(\cdot)$, respectively. Since they are class functions on $S_n$, we write $\chi_\zeta(\gamma)$ and $\tilde{\chi}_\zeta(\gamma)$ for their values on the conjugacy class $C(S_n,\gamma)$.  It is well known that $\chi_{(n)}(\sigma)=\tilde{\chi}_{(n)}(\sigma)=1$ and $\chi_{(1^n)}(\sigma)=\tilde{\chi}_{(1^n)}(\sigma)=\mathrm{sgn}(\sigma)$ for every $\sigma\in S_n$. Note that $\chi_\zeta(e)=\dim\rho_{\zeta}$ is the dimension of $\rho_\zeta\in \widehat{S_n}$ for any $\zeta\vdash n$.

In view of \eqref{eq:muj}, for a normal Cayley graph $\mathrm{Cay}(S_n, S)$ and a partition $\zeta \vdash n$, we use $\lambda_\zeta(\mathrm{Cay}(S_n, S))$, or $\lambda_\zeta$ for short, to denote its eigenvalue corresponding to the Specht module $\rho_\zeta$. That is,
\begin{equation}\label{eq:eigen-chara}
    \lambda_\zeta(\mathrm{Cay}(S_n, S)) = \sum_{\sigma \in S} \tilde{\chi}_\zeta(\sigma).
\end{equation}
It is known \cite{F} that the character of every $\rho_\zeta\in \widehat{S_n}$ on every conjugacy class of $S_n$ is an integer whose absolute value is at most the dimension of $\rho_\zeta$. Hence $\tilde{\chi}_\zeta(\gamma)$ is a rational number in the interval $[-1,1]$ for all $\zeta,\gamma\vdash n$.  The following table, which is extracted from \cite[Table 1]{PP}, gives the dimensions and characters of some irreducible representations of $S_n$.

\begin{table}[ht]
\centering
\begin{tabular}{c c c}
\toprule
$~~~~\zeta\vdash n~~~$ & $~~~~\dim \rho_\zeta=\chi_\zeta(e)~~~~$ & $~~~~\chi_\zeta(\gamma)~\text{with}~c_i(\gamma)=c_i~~~~$\\\toprule
$(n)$ & $1$ & $1$\\ \midrule
$(n-1,1)$ & $n-1$ & $c_1-1$\\\midrule
$(n-2,2)$ & $\frac{n(n-3)}{2}$ & $\frac{c_1(c_1-3)}{2}+c_2$\\\midrule
$(n-2,1,1)$& $\frac{(n-1)(n-2)}{2}$ & $\frac{(c_1-1)(c_1-2)}{2}-c_2$\\ \midrule
$(n-3,3)$& $\frac{n(n-1)(n-5)}{6}$ & $\frac{c_1(c_1-1)(c_1-5)}{6}+(c_1-1)c_2+c_3$\\ \midrule
$(n-3,2,1)$ & $\frac{n(n-2)(n-4)}{3}$ & $\frac{c_1(c_1-2)(c_1-4)}{3}-c_3$ \\ 
\bottomrule
\end{tabular}
\vspace{0.2cm}
\caption{Dimensions and characters of some irreducible representations of $S_n$}
\label{tab:tab1}
\end{table}

The following lemma will play a crucial role in our proof of Theorem~\ref{thm:main1}.

\begin{lem}[cf.~{\cite[Proposition 2.3]{PP}}]\label{PPmax}
There exists a positive integer $N_0$ such that for every $n\ge N_0$ and any $\gamma\vdash n$ with $|\fix(\gamma)| \geq 2$, we have $$\max_{\substack{\zeta\vdash n \\ \zeta \notin \{(n), (1^n)\}}} \tilde{\chi}_\zeta(\gamma)=\tilde{\chi}_{(n-1,1)}(\gamma).$$ 
\end{lem}

\section{Proof of Theorem \ref{thm:complement}}
\label{sec:complements}

\begin{proof}[Proof of Theorem~\ref{thm:complement}]
Let $n \geq 5$, and let $S$ be a nonempty normal set in $S_n$ with $S\neq S_n\setminus\{e\}$. Assume first that either $S$ or $S^c$ is equal to $A_n\setminus\{e\}$. Without loss of generality, we may assume that $S=A_n\setminus\{e\}$.
Then $\mathrm{Cay}(S_n,S)\cong 2K_{n!/2}$ has spectrum
\(
\left\{\left(\frac{n!}{2}-1\right)^2,(-1)^{n!-2}\right\}.
\)
The eigenvalue $\frac{n!}{2}-1$ is attained by both $\rho_{(n)}$ and $\rho_{(1^n)}$, while the strictly second largest eigenvalue $-1$ is attained by $\rho_{(n-1,1)}$. Since $S^c=S_n\setminus A_n$, we have $\mathrm{Cay}(S_n,S^c)\cong K_{n!/2,n!/2}$, and its spectrum is
\(
\left\{\left(\frac{n!}{2}\right)^1,0^{n!-2},\left(-\frac{n!}{2}\right)^1\right\},
\)
with the largest eigenvalue $\frac{n!}{2}$ and the smallest eigenvalue $-\frac{n!}{2}$ attained by $\rho_{(n)}$ and $\rho_{(1^n)}$, respectively. The strictly second largest eigenvalue of $\mathrm{Cay}(S_n,S^c)$ is $0$, which is attained by $\rho_{(n-1,1)}$.
Therefore, both $\mathrm{Cay}(S_n,S)$ and $\mathrm{Cay}(S_n,S^c)$ have the Aldous property.

From now on, assume that neither $S$ nor $S^c$ is equal to $A_n\setminus\{e\}$. We aim to show that at most one of $\mathrm{Cay}(S_n, S)$ and $\mathrm{Cay}(S_n, S^c)$ has the Aldous property. Without loss of generality, assume that $\mathrm{Cay}(S_n,S)$ has the Aldous property. We prove under this assumption that $\mathrm{Cay}(S_n,S^c)$ does not have the Aldous property.

The union of $\mathrm{Cay}(S_n,S)$ and $\mathrm{Cay}(S_n,S^c)$ is the complete graph $\mathrm{Cay}(S_n,S_n\setminus\{e\}) \cong K_{n!}$, which has eigenvalue $n!-1$ with multiplicity $1$ attained by the trivial representation $\rho_{(n)}$, and eigenvalue $-1$ with multiplicity $n!-1$ attained by all nontrivial irreducible representations of $S_n$. Therefore, by \eqref{eq:eigen-chara}, for each $(n)\ne\zeta\vdash n$,
\begin{equation}\label{complement-sum}
\lambda_{\zeta}(\mathrm{Cay}(S_n,S)) + \lambda_{\zeta}(\mathrm{Cay}(S_n,S^c))
=
\sum_{\sigma\in S}\tilde{\chi}_{\zeta}(\sigma)+\sum_{\sigma\in S^c}\tilde{\chi}_{\zeta}(\sigma)
=
\lambda_{\zeta}(\mathrm{Cay}(S_n,S_n\setminus\{e\}))
=-1.
\end{equation}

Since $n \ge 5$, the alternating group $A_n$ is the only non-trivial proper normal subgroup of $S_n$. Since $S$ and $S^c$ are nonempty normal sets, the subgroups $\langle S \rangle$ and $\langle S^c \rangle$ are normal in $S_n$, leaving $A_n$ and $S_n$ as the only possibilities. However, $\langle S \rangle$ and $\langle S^c \rangle$ cannot both be $A_n$, for otherwise we would have $S \cup S^c \subseteq A_n$, which contradicts $S \cup S^c = S_n \setminus \{e\}$. Therefore, exactly one of the following three cases must occur.

\smallskip
\textsf{Case 1.} $\langle S\rangle=S_n$ and $\langle S^c\rangle=S_n$.
\smallskip

In this case, both $\mathrm{Cay}(S_n,S)$ and $\mathrm{Cay}(S_n,S^c)$ are connected, and their largest eigenvalues are attained by the trivial representation $\rho_{(n)}$. Since both graphs have diameter at least $2$, it follows from \cite[Lemma 8.12.1]{GR2001} that each has at least three distinct eigenvalues. Since $\mathrm{Cay}(S_n,S)$ has the Aldous property by our assumption, its strictly second largest eigenvalue is
$\lambda_{(n-1,1)}(\mathrm{Cay}(S_n,S))$. Let $\lambda_{\beta}(\mathrm{Cay}(S_n,S))$ be the smallest eigenvalue of $\mathrm{Cay}(S_n,S)$, where $\beta\vdash n$ satisfies $\beta\notin\{(n),(n-1,1)\}$. Then
$$
\lambda_{(n-1,1)}(\mathrm{Cay}(S_n,S)) > \lambda_{\beta}(\mathrm{Cay}(S_n,S)).
$$
This together with \eqref{complement-sum} yields
\begin{align*}
\lambda_{\beta}(\mathrm{Cay}(S_n,S^c))
&=-\lambda_{\beta}(\mathrm{Cay}(S_n,S))-1 \\
&> -\lambda_{(n-1,1)}(\mathrm{Cay}(S_n,S))-1 \\
&= \lambda_{(n-1,1)}(\mathrm{Cay}(S_n,S^c)).
\end{align*}
Therefore, $\mathrm{Cay}(S_n,S^c)$ does not have the Aldous property.

\smallskip
\textsf{Case 2.} $\langle S\rangle=S_n$ and $\langle S^c\rangle=A_n$.
\smallskip

In this case, $\mathrm{Cay}(S_n,S)$ is connected, and $\mathrm{Cay}(S_n,S^c)$ is disconnected with two isomorphic connected components. The largest eigenvalue $|S|$ of $\mathrm{Cay}(S_n,S)$ has multiplicity $1$ and is attained by $\rho_{(n)}$. The largest eigenvalue $|S^c|$ of $\mathrm{Cay}(S_n,S^c)$ has multiplicity $2$ and is attained by both $\rho_{(n)}$ and $\rho_{(1^n)}$.

Since $S\neq S_n\setminus\{e\}$,  $\mathrm{Cay}(S_n,S)$ has diameter at least $2$. So, by \cite[Lemma 8.12.1]{GR2001}, $\mathrm{Cay}(S_n,S)$ has at least three distinct eigenvalues. Since $\mathrm{Cay}(S_n,S)$ has the Aldous property by our assumption, its strictly second largest eigenvalue is $\lambda_{(n-1,1)}(\mathrm{Cay}(S_n,S)).$

By \eqref{complement-sum}, among all partitions $\zeta\vdash n$ with $\zeta\ne (n)$, any representation attaining the largest eigenvalue of $\mathrm{Cay}(S_n,S^c)$ also attains the smallest eigenvalue of $\mathrm{Cay}(S_n,S)$. Since the largest eigenvalue $|S^c|$ of $\mathrm{Cay}(S_n,S^c)$ is attained by $\rho_{(1^n)}$, it follows that the smallest eigenvalue of $\mathrm{Cay}(S_n,S)$ is attained by the sign representation $\rho_{(1^n)}$.

If $\mathrm{Cay}(S_n,S)$ has at least four distinct eigenvalues, then there exists $\beta\vdash n$ with $\beta\notin\{(n),(n-1,1),(1^n)\}$ such that
\[
\lambda_{(n-1,1)}(\mathrm{Cay}(S_n,S)) > \lambda_{\beta}(\mathrm{Cay}(S_n,S)) > \lambda_{(1^n)}(\mathrm{Cay}(S_n,S)).
\]
This together with \eqref{complement-sum} implies that
\begin{align*}
\lambda_{\beta}(\mathrm{Cay}(S_n,S^c))
&=-\lambda_{\beta}(\mathrm{Cay}(S_n,S))-1 \\
&> -\lambda_{(n-1,1)}(\mathrm{Cay}(S_n,S))-1 \\
&= \lambda_{(n-1,1)}(\mathrm{Cay}(S_n,S^c)).
\end{align*}
Hence $\mathrm{Cay}(S_n,S^c)$ does not have the Aldous property.

If $\mathrm{Cay}(S_n,S)$ has exactly three distinct eigenvalues, then by \eqref{complement-sum} there are exactly two distinct values among the eigenvalues $\lambda_{\beta}(\mathrm{Cay}(S_n,S^c))$ with $\beta\neq (n)$. As mentioned above,
\(
\lambda_{(n)}(\mathrm{Cay}(S_n,S^c))=\lambda_{(1^n)}(\mathrm{Cay}(S_n,S^c)).
\)
Therefore, $\mathrm{Cay}(S_n,S^c)$ has exactly two distinct eigenvalues. Thus $\mathrm{Cay}(S_n,S^c)$ is the disjoint union of two cliques, which forces $S^c=A_n\setminus\{e\}$, contradicting our assumption. 

\smallskip
\textsf{Case 3.} $\langle S\rangle=A_n$ and $\langle S^c\rangle=S_n$.
\smallskip

In this case, $\mathrm{Cay}(S_n,S)$ is disconnected with two isomorphic connected components, and  $\mathrm{Cay}(S_n,S^c)$ is connected. The largest eigenvalue $|S|$ of $\mathrm{Cay}(S_n,S)$ has multiplicity $2$ and is attained by both $\rho_{(n)}$ and $\rho_{(1^n)}$. The largest eigenvalue $|S^c|$ of $\mathrm{Cay}(S_n,S^c)$ has multiplicity $1$ and is attained by $\rho_{(n)}$.

Since $S\neq A_n\setminus\{e\}$, $\mathrm{Cay}(S_n,S)$ is not the disjoint union of two cliques. Hence each connected component of $\mathrm{Cay}(S_n,S)$ has diameter at least $2$. Again, by \cite[Lemma 8.12.1]{GR2001}, $\mathrm{Cay}(S_n,S)$ has at least three distinct eigenvalues. Since $\mathrm{Cay}(S_n,S)$ has the Aldous property by our assumption, its strictly second largest eigenvalue is 
$\lambda_{(n-1,1)}(\mathrm{Cay}(S_n,S))$. Let $\lambda_{\beta}(\mathrm{Cay}(S_n,S))$ be the smallest eigenvalue of $\mathrm{Cay}(S_n,S)$, where $\beta\vdash n$ satisfies $\beta\notin\{(n),(n-1,1),(1^n)\}$. Then
\[
\lambda_{(n-1,1)}(\mathrm{Cay}(S_n,S)) > \lambda_{\beta}(\mathrm{Cay}(S_n,S)).
\]
Combining this with \eqref{complement-sum}, we obtain that
\begin{align*}
\lambda_{\beta}(\mathrm{Cay}(S_n,S^c))
&=-\lambda_{\beta}(\mathrm{Cay}(S_n,S))-1 \\
&> -\lambda_{(n-1,1)}(\mathrm{Cay}(S_n,S))-1 \\ 
&= \lambda_{(n-1,1)}(\mathrm{Cay}(S_n,S^c)).
\end{align*}
Hence $\mathrm{Cay}(S_n,S^c)$ does not have the Aldous property.

This proves Theorem~\ref{thm:complement}. 
\end{proof}

\section{Proofs of Theorem~\ref{thm:main1} and Corollary~\ref{An-sec-eigen}}
\label{supp-small}

\begin{proof}[Proof of Theorem~\ref{thm:main1}]
Let $N_0$ be the positive integer in Lemma~\ref{PPmax}. Let $N = \max\{N_0, 5\}$ and $n\ge N$. Let $S$ be a nonempty normal set in $S_n$ such that $2 \le |\mathrm{supp}(\sigma)| \le n-2$ for all  $\sigma\in S$. Since $S$ is normal, $\langle S\rangle$ is a nontrivial normal subgroup of $S_n$ and hence is either $S_n$ or $A_n$. Since $\mathrm{Cay}(S_n,S)$ is normal, by \eqref{eq:eigen-chara}, its eigenvalue corresponding to a partition $\beta\vdash n$ is
\[
\lambda_\beta
=
\sum_{\sigma \in S} \tilde{\chi}_\beta(\sigma).
\]

(a) Assume first that $O(S)=\emptyset$. Then $S\subseteq A_n$ and hence $\langle S\rangle=A_n$. Thus $\mathrm{Cay}(S_n,S)$ is disconnected with two isomorphic connected components. Consequently, the largest eigenvalue $|S|$ of $\mathrm{Cay}(S_n,S)$ has multiplicity $2$ and is attained by both $\rho_{(n)}$ and $\rho_{(1^n)}$. Therefore, the strictly second largest eigenvalue of $\mathrm{Cay}(S_n,S)$ can be expressed as
\begin{equation}\label{max-even}
\max_{\substack{\beta\vdash n\\ \beta\notin\{(n),(1^n)\}}}\lambda_\beta
=
\max_{\substack{\beta\vdash n\\ \beta\notin\{(n),(1^n)\}}}
\left(
\sum_{\sigma \in S} \tilde{\chi}_\beta(\sigma)
\right).
\end{equation}
By Lemma~\ref{PPmax}, the maximum on the right-hand side of \eqref{max-even} is attained by $\beta=(n-1,1)$. Therefore, $\mathrm{Cay}(S_n, S)$ has the Aldous property.

Now assume that $E(S)=\emptyset$. Then every permutation in $S$ is odd, and hence $\langle S\rangle=S_n$. Therefore, $\mathrm{Cay}(S_n, S)$ is connected, and its largest eigenvalue $|S|$ has multiplicity $1$ and is attained by the trivial representation $\rho_{(n)}$.
Furthermore, the eigenvalue of $\mathrm{Cay}(S_n, S)$ corresponding to the sign representation $\rho_{(1^n)}$ is $-|S|$. So, by the Perron--Frobenius theorem \cite[Theorem 8.4.4]{HJ1}, $-|S|$ is the smallest eigenvalue of $\mathrm{Cay}(S_n, S)$. Therefore, the strictly second largest eigenvalue of $\mathrm{Cay}(S_n, S)$ is again given by \eqref{max-even}. Again, by Lemma~\ref{PPmax}, the maximum on the right-hand side of \eqref{max-even} is attained by $\beta=(n-1,1)$ and therefore $\mathrm{Cay}(S_n, S)$ has the Aldous property.

(b) Assume that $E(S)\ne\emptyset$ and $O(S)\ne\emptyset$. Then $\langle S\rangle=S_n$ and hence $\mathrm{Cay}(S_n, S)$ is connected. The largest eigenvalue $|S|$ of $\mathrm{Cay}(S_n,S)$ has multiplicity $1$ and is attained by the trivial representation $\rho_{(n)}$. The eigenvalue of $\mathrm{Cay}(S_n,S)$ corresponding to the sign representation $\rho_{(1^n)}$ is
\begin{equation}
\label{eq:sign}
\lambda_{(1^n)}
=
\sum_{\sigma\in E(S)}\tilde{\chi}_{(1^n)}(\sigma)
+
\sum_{\sigma\in O(S)}\tilde{\chi}_{(1^n)}(\sigma)
=
|E(S)|-|O(S)|.
\end{equation}
We have
\begin{equation}\label{max-mix}
\max_{\substack{\beta\vdash n\\ \beta\notin\{(n),(1^n)\}}}\lambda_\beta
=
\max_{\substack{\beta\vdash n\\ \beta\notin\{(n),(1^n)\}}}
\left(
\sum_{\sigma \in S} \tilde{\chi}_\beta(\sigma)
\right).
\end{equation}
By Lemma~\ref{PPmax}, the maximum on the right-hand side of \eqref{max-mix} is attained by $\beta=(n-1,1)$. Thus,
\begin{equation}
\label{eq:two}
\lambda_2(\mathrm{Cay}(S_n, S))=\max\{\lambda_{(n-1,1)},\lambda_{(1^n)}\}.
\end{equation}
It follows that $\mathrm{Cay}(S_n, S)$ has the Aldous property if and only if $\lambda_{(n-1,1)} \geq \lambda_{(1^n)}$. Since $\lambda_{(1^n)}
= |E(S)|-|O(S)|$ and 
$$
\lambda_{(n-1,1)}
=
\sum_{\sigma\in S}\tilde{\chi}_{(n-1,1)}(\sigma)
=
\sum_{\sigma\in S}\frac{|\fix(\sigma)|-1}{n-1},
$$
we obtain that $\mathrm{Cay}(S_n, S)$ has the Aldous property if and only if $$
\sum_{\sigma\in S}\frac{|\fix(\sigma)|-1}{n-1} \ge |E(S)| -|O(S)|,
$$
or, equivalently,
$$
\sum_{\sigma\in S} |\fix(\sigma)| \ge n|E(S)|-(n-2)|O(S)|.
$$
Moreover, by \eqref{eq:two}, if $\mathrm{Cay}(S_n, S)$ does not have the Aldous property, then its strictly second largest eigenvalue is attained by the sign representation $\rho_{(1^n)}$ of $S_n$. 
\end{proof}

\begin{proof}[Proof of Corollary~\ref{An-sec-eigen}]
Let $N \ge 5$ be the positive integer in Theorem~\ref{thm:main1}. Let $n\geq N$, and let $S$ be a nonempty normal set in $S_n$ such that $S\subseteq A_n$ and $2\leq |\mathrm{supp}(\sigma)|\leq n-2$ for every $\sigma\in S$.

Since $S\subseteq A_n$, part (a) of Theorem~\ref{thm:main1} implies that $\mathrm{Cay}(S_n,S)$ has the Aldous property. Hence its strictly second largest eigenvalue is attained by $\rho_{(n-1,1)}$. As $\mathrm{Cay}(S_n,S)$ is normal, \eqref{eq:eigen-chara} and Table~\ref{tab:tab1} show that this eigenvalue is
\[
\lambda_{(n-1,1)}
=
\sum_{\sigma \in S} \tilde{\chi}_{(n-1,1)}(\sigma)
=
\sum_{\sigma \in S} \frac{|\fix(\sigma)|-1}{n-1}
=
\frac{1}{n-1} \left( \sum_{\sigma \in S} |\fix(\sigma)| - |S| \right).
\]

Since $S$ is a nonempty normal set consisting of nonidentity elements, $\langle S\rangle$ is a nontrivial normal subgroup of $S_n$. Moreover, $\langle S\rangle\subseteq A_n$. As $n\geq 5$, it follows that $\langle S\rangle=A_n$. Therefore, $\mathrm{Cay}(S_n,S)$ has exactly two connected components, induced on $A_n$ and $S_n\setminus A_n$, respectively, and each component is isomorphic to $\mathrm{Cay}(A_n,S)$.
\end{proof}

\section{Normal Cayley graphs on symmetric groups with strictly second largest eigenvalue at most 1}\label{line-graph}

This section is a preparation for our proof of Theorem \ref{thm:main2} to be given in the next section. However, the main results in this section, Theorems \ref{class-linegraph} and \ref{sn-cay-sec-1}, are also of interest for their own sake. 

\subsection{Normal Cayley graphs on $S_n$ that are line graphs}

The \emph{line graph} of a graph $\Gamma$ is the graph whose vertices correspond to the edges of $\Gamma$, where two vertices are adjacent if and only if the corresponding edges in $\Gamma$ share exactly one end-vertex. A graph is called a \emph{line graph} if it is isomorphic to the line graph of some graph. The following result gives a classification of normal Cayley graphs on $S_n$ that are line graphs. 

\begin{thm}
\label{class-linegraph}
Let $n\ge 2$, and let $S$ be a nonempty normal set in $S_n$. Then $\mathrm{Cay}(S_n,S)$ is a line graph if and only if one of the following holds:
	\begin{enumerate}[\rm (a)]
		\item $S=S_n\setminus\{e\}$, and hence $\mathrm{Cay}(S_n,S)\cong K_{n!}$;
		\item $S=A_n\setminus\{e\}$, and hence $\mathrm{Cay}(S_n,S)\cong 2K_{n!/2}$;
		\item $n=4$ and $S=\{(1\,2)(3\,4), (1\,3)(2\,4), (1\,4)(2\,3)\}$, and hence $\mathrm{Cay}(S_n,S)\cong 6K_4$.
	\end{enumerate}
\end{thm}

Our proof of Theorem \ref{class-linegraph} is based on Beineke's characterization of line graphs. In 1970, he proved \cite{Beineke1970} that a graph is a line graph if and only if it does not contain any one of the nine forbidden graphs as an induced subgraph. For our purposes, we only show two of these nine forbidden subgraphs, denoted by $H_1$ and $H_2$, in Figure~\ref{fig:H1H2}. 

\begin{figure}[htbp]
    \centering
    \captionsetup[subfigure]{labelformat=empty}
    \begin{subfigure}[b]{0.45\textwidth}
    \centering
    \begin{tikzpicture}[dot/.style={circle, fill=black, inner sep=1.5pt}, scale=1.2, transform shape]
        \node[dot, label=left:{$\delta$}] (e) at (0,0) {}; 
        \node[dot, label=above:{$\tau_1$}] (t1) at (1.5, 1) {};
        \node[dot, label=right:{$\tau_2$}] (t2) at (2, 0) {};
        \node[dot, label=below:{$\tau_3$}] (t3) at (1.5, -1) {};
        \node[dot, label=right:{$\omega$}] (s) at (3.5, 0) {}; 
        
        \draw (t1)--(t2)--(t3)--(t1); 
        \draw (e)--(t1) (e)--(t2) (e)--(t3); 
        \draw (s)--(t1) (s)--(t2) (s)--(t3); 
    \end{tikzpicture}
    \caption{$H_1$}
    \label{fig:sub-a}
\end{subfigure}
    \hfill 
    \begin{subfigure}[b]{0.45\textwidth}
        \centering
        \begin{tikzpicture}[dot/.style={circle, fill=black, inner sep=1.5pt}, scale=1.2, transform shape]
            \node[dot, label=left:{$\delta$}] (e) at (0,0) {};
            \node[dot, label=right:{$\tau_1$}] (t1) at (1.5, 1) {};
            \node[dot, label=right:{$\tau_2$}] (t2) at (1.5, 0) {};
            \node[dot, label=right:{$\tau_3$}] (t3) at (1.5, -1) {};
            
            \draw (e)--(t1);
            \draw (e)--(t2);
            \draw (e)--(t3);
        \end{tikzpicture}
        \caption{$H_2$}
        \label{fig:sub-b}
    \end{subfigure}

    \vfill 
    \caption{Two forbidden subgraphs for line graphs}
    \label{fig:H1H2}
\end{figure}
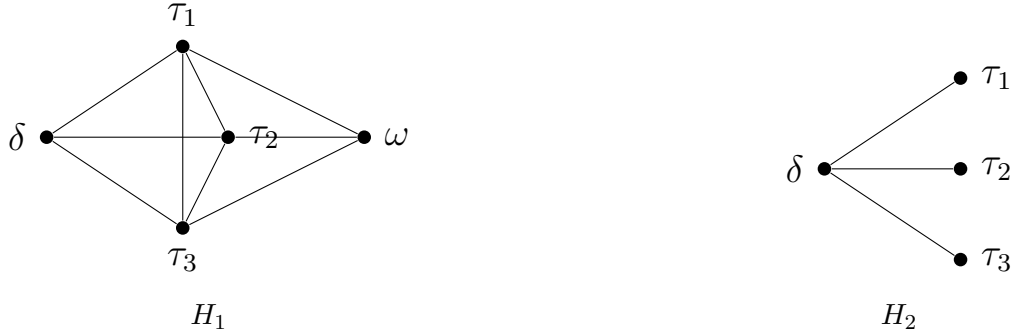

The following lemma deals with the case when the connection set contains all transpositions. In its proof we write $x\sim y$ to denote that $x$ and $y$ are adjacent in the Cayley graph $\mathrm{Cay}(S_n, S)$ under consideration.

\begin{lem}\label{2-cycle-linegraph}
Let $n \ge 2$, and let $S$ be a nonempty normal set in $S_n$ which contains the conjugacy class of all transpositions. Then $\mathrm{Cay}(S_n, S)$ is a line graph if and only if $S = S_n \setminus \{e\}$, or equivalently, $\mathrm{Cay}(S_n, S) \cong K_{n!}$.
\end{lem}

\begin{proof}
The sufficiency is obvious. It remains to prove the necessity.

Suppose that $\mathrm{Cay}(S_n, S)$ is a line graph. Then by Beineke's characterization it does not contain $H_1$ or $H_2$ (see Figure~\ref{fig:H1H2}) as an induced subgraph.
We aim to prove that every permutation $\sigma \in S_n$ with $|\mathrm{supp}(\sigma)| \ge 2$ belongs to $S$. We prove this statement by induction on $|\mathrm{supp}(\sigma)|$. If $|\mathrm{supp}(\sigma)|=2$, then the statement follows from the assumption that $S$ contains all transpositions. If $|\mathrm{supp}(\sigma)|=3$, then $\sigma \in S$, for otherwise an induced subgraph of $\mathrm{Cay}(S_n, S)$ isomorphic to $H_2$ can be easily found by taking $\delta=e$, $\tau_1=(1\,2)$, $\tau_2=(2\,3)$, and $\tau_3=(1\,3)$ in Figure~\ref{fig:H1H2}.

Assume inductively that for some $k \ge 4$ every permutation in $S_n$ with support size at most $k-1$ belongs to $S$. We now show that for an arbitrary permutation $\sigma \in S_n$ with $|\mathrm{supp}(\sigma)|=k$ we must have $\sigma \in S$. Suppose first that the cycle decomposition of $\sigma$ contains a cycle of length at least $3$. Then we may write without loss of generality that
\(
\sigma=(1\,2\,\cdots\,t)\sigma'
\)
for some $t\ge 3$ and $\sigma' \in S_n$. Set
\[
\delta=e,\qquad
\tau_1=(2\,3\,\cdots\,t)\sigma',\qquad
\tau_2=(1\,3\,\cdots\,t)\sigma',\qquad
\tau_3=(3\,\cdots\,t)\sigma',\qquad
\omega=\sigma.
\]
By the induction hypothesis, $e \sim \tau_i$ for $1\le i\le 3$. Moreover, a direct computation shows that $\tau_i^{-1}\tau_j$ for $1\le i<j\le 3$ and $\tau_i^{-1}\sigma$ for $1\le i\le 3$ are $2$-cycles or $3$-cycles. Since by our induction hypothesis all $2$-cycles and $3$-cycles belong to $S$, we have $\tau_i \sim \sigma$ for $1\le i\le 3$ and $\tau_i\sim\tau_j$ for all $1\le i<j\le 3$. Since $\mathrm{Cay}(S_n,S)$ does not contain an induced subgraph isomorphic to $H_1$, we must have $e\sim\sigma$ and therefore $\sigma\in S$.

It remains to consider the case when the cycle decomposition of $\sigma$ contains only cycles of lengths $1$ and $2$. Since $|\mathrm{supp}(\sigma)| = k \ge 4$, we may assume without loss of generality that
\(
\sigma=(1\,2)(3\,4)\sigma'
\)
for some $\sigma' \in S_n$. Let
\[
\delta=e,\qquad
\tau_1=(1\,2\,3)\sigma',\qquad
\tau_2=(1\,2\,3\,4)\sigma',\qquad
\tau_3=(1\,2\,4\,3)\sigma',\qquad
\omega=\sigma.
\]
By what we proved in the previous paragraph, we have $e\sim\tau_i$ for $1\le i\le 3$. Moreover, a direct computation shows that $\tau_i^{-1}\tau_j$ for $1\le i<j\le 3$ and $\tau_i^{-1}\sigma$ for $1\le i\le 3$ have supports of size at most $3$. Since by our induction hypothesis all $2$-cycles and $3$-cycles belong to $S$, we have $\tau_i \sim \sigma$ for $1\le i\le 3$ and $\tau_i\sim\tau_j$ for all $1\le i<j\le 3$. Since $\mathrm{Cay}(S_n,S)$ does not contain an induced subgraph isomorphic to $H_1$, we must have $e\sim\sigma$ and hence $\sigma\in S$.

So, by mathematical induction, we have proved that $S$ contains every permutation in $S_n$ with support size at least $2$. Therefore,
\(
S=S_n\setminus\{e\}, 
\)
or equivalently,
\(
\mathrm{Cay}(S_n,S) \cong K_{n!}.
\)
\end{proof}

The next lemma deals with the case when the connection set does not contain the conjugacy class of transpositions but contains the conjugacy class of $3$-cycles.

\begin{lem}\label{3-cycle-linegraph}
Let $n\ge 3$, and let $S$ be a nonempty normal set in $S_n$ which contains the conjugacy class of $3$-cycles but not the conjugacy class of transpositions. Then $\mathrm{Cay}(S_n,S)$ is a line graph 
if and only if $S=A_n\setminus\{e\}$, in which case $\mathrm{Cay}(S_n,S) \cong 2 K_{n!/2}$.
\end{lem}
\begin{proof}
The sufficiency is obvious. We prove the necessity.

Suppose that $\mathrm{Cay}(S_n, S)$ is a line graph. Then by Beineke's characterization it does not contain $H_1$ or $H_2$ (see Figure~\ref{fig:H1H2}) as an induced subgraph. We first show that
\(
A_n\setminus\{e\}\subseteq S,
\)
that is, every permutation $\sigma\in A_n$ with $|\mathrm{supp}(\sigma)|\ge 3$ belongs to $S$. We prove this statement by induction on $|\mathrm{supp}(\sigma)|$. Clearly, if $|\mathrm{supp}(\sigma)|=3$, the statement follows from our assumption that $S$ contains all 3-cycles. 

Let $\sigma \in A_n$ be such that $|\mathrm{supp}(\sigma)|=4$. Since $\sigma$ is an even permutation, without loss of generality we may assume that
\(
\sigma=(1\,2)(3\,4).
\)
Suppose, for a contradiction, that $\sigma \notin S$. Since $S$ is a union of conjugacy classes, this assumption implies that no permutation with cycle type $(2,2, 1, \ldots, 1)$ belongs to $S$. 
Let
\[
\delta=e,\qquad
\tau_1=(1\,2\,3),\qquad
\tau_2=(1\,3\,4),\qquad
\tau_3=(2\,4\,3).
\]
Then, by our assumption, $e \sim \tau_i$ for $1\le i\le 3$. A direct calculation leads to
\begin{align*}
\tau_1^{-1}\tau_2 =  (1\,4)(2\,3), \qquad
\tau_1^{-1}\tau_3 =  (1\,2)(3\,4), \qquad
\tau_2^{-1}\tau_3 =  (1\,3)(2\,4).
\end{align*}
Therefore, $\tau_1$, $\tau_2$, and $\tau_3$ are mutually nonadjacent in $\mathrm{Cay}(S_n,S)$.
Consequently, the subgraph induced by $\{e, \tau_1, \tau_2, \tau_3\}$ is isomorphic to the forbidden subgraph $H_2$ in Figure~\ref{fig:H1H2}, a contradiction. Hence $\sigma \in S$.

Assume inductively that for some $k\ge 5$ every even permutation with support size at most $k-1$ belongs to $S$. Let $\sigma\in A_n$ be such that $|\mathrm{supp}(\sigma)|=k$. We aim to show that $\sigma\in S$. If the cycle decomposition of $\sigma$ contains an odd cycle of length at least $5$, then we may assume without loss of generality that
\(
\sigma=(1\,2\,3\,\cdots\,t)\sigma'
\)
for some odd $t\ge 5$ and $\sigma' \in A_n$. Note that $t \le k$. Let
\[
\delta=e,\qquad
\tau_1=(1\,2\,3)\sigma',\qquad
\tau_2=(2\,3\,4)\sigma',\qquad
\tau_3=(1\,3\,4)\sigma',\qquad
\omega=\sigma.
\]
Then each $\tau_i$ is an even permutation with support size at most $k-2$, and hence $e\sim\tau_i$ for $1\le i\le 3$ by the induction hypothesis. Moreover, each $\tau_i^{-1}\omega$ is an even permutation with support size at most $t-1$ ($\le k-1$), and so $\tau_i \sim \sigma$ for $1\le i\le 3$. Also, each $\tau_i^{-1}\tau_j$ is an even permutation with support size at most $4$, and hence $\tau_i\sim\tau_j$ for $1\le i<j\le 3$ by the induction hypothesis. Since $\mathrm{Cay}(S_n,S)$ is a line graph, it does not contain $H_1$ as an induced subgraph. Thus, we must have $e \sim \sigma$, that is, $\sigma\in S$.

If the cycle decomposition of $\sigma$ contains an even cycle of length at least $4$, then we may assume without loss of generality that
\(
\sigma=(1\,2\,3\,\cdots\,t)\sigma'
\)
for some even $t\ge 4$ and $\sigma' \in S_n$, and we set
\[
\delta=e,\qquad
\tau_1=(1\,2)\sigma',\qquad
\tau_2=(2\,3)\sigma',\qquad
\tau_3=(1\,3)\sigma',\qquad
\omega=\sigma.
\]
If the cycle decomposition of $\sigma$ contains both a $2$-cycle and a $3$-cycle, then we may assume without loss of generality that
\(
\sigma=(1\,2\,3)(4\,5)\sigma'
\)
for some $\sigma' \in S_n$, and we set
\[
\delta=e,\qquad
\tau_1=(1\,2)\sigma',\qquad
\tau_2=(2\,3)\sigma',\qquad
\tau_3=(1\,3)\sigma',\qquad
\omega=\sigma.
\]
If the cycle decomposition of $\sigma$ contains only $3$-cycles, then since $|\mathrm{supp}(\sigma)|=k\ge 5$ we may assume without loss of generality that
\(
\sigma=(1\,2\,3)(4\,5\,6)\sigma'
\)
for some $\sigma' \in A_n$, and we set 
\[
\delta=e,\qquad
\tau_1=(1\,2)(4\,5)\sigma',\qquad
\tau_2=(2\,3)(4\,5)\sigma',\qquad
\tau_3=(1\,3)(4\,5)\sigma',\qquad
\omega=\sigma.
\]
Finally, if the cycle decomposition of $\sigma$ contains only $2$-cycles, then since $|\mathrm{supp}(\sigma)|=k\ge 5$ we may assume without loss of generality that
\(
\sigma=(1\,2)(3\,4)(5\,6)\sigma'
\)
for $\sigma' \in S_n$, and we set
\[
\delta=e,\qquad
\tau_1=(1\,2)\sigma',\qquad
\tau_2=(3\,4)\sigma',\qquad
\tau_3=(5\,6)\sigma',\qquad
\omega=\sigma.
\]
In each case, arguing exactly as in the previous paragraph, we obtain that $\sigma\in S$.

So, by mathematical induction, we have proved that
\(
A_n\setminus\{e\}\subseteq S.
\)
To prove 
\(
S=A_n\setminus\{e\},
\)
it suffices to show that no odd permutation $\sigma\in S_n$ with $|\mathrm{supp}(\sigma)|\ge 4$ belongs to $S$. Assume for a contradiction that $\sigma\in S$ is odd and $|\mathrm{supp}(\sigma)|\ge 4$. Since $|\mathrm{supp}(\sigma)|\ge 4$, we can find two distinct permutations $\sigma_1$ and $\sigma_2$ in the conjugacy class $C(S_n,\sigma)$, both different from $\sigma$. Since $\sigma\in S$ and $S$ is a normal set, we have $\sigma_1, \sigma_2\in S$.
Let
\[
\delta=e,\qquad
\tau_1=\sigma,\qquad
\tau_2=\sigma_1,\qquad
\tau_3=\sigma_2,\qquad
\omega=(1\,2).
\]
Then $\delta\sim\tau_i$ for $1\le i\le 3$. Moreover, each $\tau_i^{-1}\omega$ is an even permutation, and hence $\tau_i\sim\omega$ for $1\le i\le 3$. Also, each $\tau_i^{-1}\tau_j$ is an even permutation, and hence $\tau_i\sim\tau_j$ for $1\le i<j\le 3$. On the other hand, since by our assumption the transposition $(1\,2)$ does not belong to $S$, $\delta$ and $\omega$ are not adjacent in $\mathrm{Cay}(S_n,S)$. Therefore, the subgraph of $\mathrm{Cay}(S_n,S)$ induced by $\{\delta,\tau_1,\tau_2,\tau_3,\omega\}$ is isomorphic to $H_1$, a contradiction. Therefore, 
\(
S=A_n\setminus\{e\}
\)
and consequently $\mathrm{Cay}(S_n,S) \cong 2 K_{n!/2}$. 
\end{proof}

We are now ready to prove Theorem~\ref{class-linegraph}.

\begin{proof}[Proof of Theorem~\ref{class-linegraph}]
The sufficiency is obvious, so it remains to prove the necessity.

Suppose that $S$ is a nonempty normal set in $S_n$ such that $\mathrm{Cay}(S_n,S)$ is a line graph. If $S$ contains the conjugacy class of transpositions, then by Lemma~\ref{2-cycle-linegraph}, $S=S_n\setminus\{e\}$, and hence $\mathrm{Cay}(S_n,S)\cong K_{n!}$. If $S$ does not contain the conjugacy class of transpositions but contains the conjugacy class of $3$-cycles, then by Lemma~\ref{3-cycle-linegraph}, $S=A_n\setminus\{e\}$, and hence $\mathrm{Cay}(S_n,S)\cong 2K_{n!/2}$. Therefore, we may assume in the sequel that $|\mathrm{supp}(\sigma)|\ge 4$ for every $\sigma\in S$. 

If $S$ contains an element $\sigma$ whose cycle decomposition contains a cycle of length at least $4$, then we may assume without loss of generality that
\(
\sigma=(1\,2\,3\,4\,\cdots\,t)\sigma'
\)
for some $t\ge 4$ and $\sigma' \in S_n$. Let
\[
\delta=e,\qquad
\tau_1=(1\,2\,3\,4\,\cdots\,t)\sigma',\qquad
\tau_2=(1\,3\,4\,\cdots\,t\,2)\sigma',\qquad
\tau_3=(1\,3\,2\,4\,\cdots\,t)\sigma'.
\]
Since $\tau_1$, $\tau_2$, and $\tau_3$ all lie in the conjugacy class of $\sigma$, we have $\tau_i\in S$ for $1\le i\le 3$, and hence $\delta\sim\tau_i$ for $1\le i\le 3$. Moreover, each $\tau_i^{-1}\tau_j$ is a $3$-cycle, and hence by our assumption $\tau_i$ and $\tau_j$ are not adjacent in $\mathrm{Cay}(S_n,S)$ for $1\le i<j\le 3$. Therefore, the subgraph induced by $\{\delta,\tau_1,\tau_2,\tau_3\}$ is isomorphic to $H_2$ in Figure~\ref{fig:H1H2}, but this contradicts our assumption that $\mathrm{Cay}(S_n,S)$ is a line graph.

If $S$ contains an element $\sigma$ whose cycle decomposition contains both a $2$-cycle and a $3$-cycle, then we may assume without loss of generality that
\(
\sigma=(1\,2)(3\,4\,5)\sigma'
\)
for some $\sigma' \in S_n$. Let
\[
\delta=e,\qquad
\tau_1=(1\,2)(3\,4\,5)\sigma',\qquad
\tau_2=(1\,2)(3\,5\,4)\sigma',\qquad
\tau_3=(3\,5)(1\,4\,2)\sigma'.
\]
Arguing as above, we have $\delta\sim\tau_i$ for $1\le i\le 3$ and $\tau_i$ and $\tau_j$ are not adjacent in $\mathrm{Cay}(S_n,S)$ for $1\le i<j\le 3$. Therefore, the subgraph of $\mathrm{Cay}(S_n,S)$ induced by $\{\delta,\tau_1,\tau_2,\tau_3\}$ is isomorphic to $H_2$ in Figure~\ref{fig:H1H2}, contradicting the assumption that $\mathrm{Cay}(S_n,S)$ is a line graph.

If $S$ contains an element $\sigma$ whose cycle decomposition contains only $3$-cycles, then, since $|\mathrm{supp}(\sigma)|\ge 4$, we may assume without loss of generality that
\(
\sigma=(1\,2\,3)(4\,5\,6)\sigma'
\)
for some $\sigma' \in S_n$. Let
\[
\delta=e,\qquad
\tau_1=(1\,2\,3)(4\,5\,6)\sigma',\qquad
\tau_2=(1\,6\,2)(3\,4\,5)\sigma',\qquad
\tau_3=(1\,4\,2)(3\,5\,6)\sigma'.
\]
Then $\delta\sim\tau_i$ for $1\le i\le 3$. Moreover, each $\tau_i^{-1}\tau_j$ is a $5$-cycle. Since we have already shown that no $5$-cycle lies in $S$, it follows that $\tau_i$ and $\tau_j$ are not adjacent in $\mathrm{Cay}(S_n,S)$ for $1\le i<j\le 3$. Therefore, the subgraph of $\mathrm{Cay}(S_n,S)$ induced by $\{\delta,\tau_1,\tau_2,\tau_3\}$ is isomorphic to $H_2$ in Figure~\ref{fig:H1H2}, a contradiction.

If $n\ge 5$ and $S$ contains an element $\sigma$ with at least one fixed point whose cycle decomposition contains only $2$-cycles, then, since $|\mathrm{supp}(\sigma)|\ge 4$, we may assume without loss of generality that $5 \notin \mathrm{supp}(\sigma)$ and $\sigma=(1\,2)(3\,4)\sigma'$ for some $\sigma' \in S_n$. Let
$$
\delta=e,\qquad
\tau_1=(1\,2)(3\,4)\sigma',\qquad
\tau_2=(1\,5)(3\,4)\sigma',\qquad
\tau_3=(2\,5)(3\,4)\sigma'.
$$
Arguing as above, we have $\delta\sim\tau_i$ for $1\le i\le 3$ and $\tau_i$ and $\tau_j$ are not adjacent in $\mathrm{Cay}(S_n,S)$ for $1\le i<j\le 3$. Therefore, the subgraph of $\mathrm{Cay}(S_n,S)$ induced by $\{\delta,\tau_1,\tau_2,\tau_3\}$ is isomorphic to $H_2$ in Figure~\ref{fig:H1H2}, and hence $\mathrm{Cay}(S_n,S)$ cannot be a line graph, a contradiction.

If $n\ge 5$ and $S$ contains an element $\sigma$ with no fixed point whose cycle decomposition consists only of $2$-cycles, then, since $n\ge 5$, we may assume without loss of generality that
$\sigma=(1\,2)(3\,4)(5\,6)\sigma'$ for some $\sigma' \in S_n$. Let
$$
\delta=e,\qquad
\tau_1=(1\,2)(3\,4)(5\,6)\sigma',\qquad
\tau_2=(1\,3)(2\,5)(4\,6)\sigma',\qquad
\tau_3=(1\,4)(2\,6)(3\,5)\sigma'.
$$
Then $\delta\sim\tau_i$ for $1\le i\le 3$. Moreover, the cycle decomposition of each $\tau_i^{-1}\tau_j$ consists of two $3$-cycles for $1\le i<j\le 3$. Since we have already shown that no element of $S$ contains a $3$-cycle, it follows that $\tau_i$ and $\tau_j$ are not adjacent in $\mathrm{Cay}(S_n,S)$ for $1\le i<j\le 3$. Therefore, the subgraph of $\mathrm{Cay}(S_n,S)$ induced by $\{\delta,\tau_1,\tau_2,\tau_3\}$ is isomorphic to $H_2$ in Figure~\ref{fig:H1H2}, and hence $\mathrm{Cay}(S_n,S)$ cannot be a line graph, a contradiction.

Summing up, the only remaining case is that $n=4$ and $S$ contains an element $\sigma$ whose cycle decomposition contains only $2$-cycles. In this case, we have $S=\{(1\,2)(3\,4), \allowbreak (1\,3)(2\,4), \allowbreak (1\,4)(2\,3)\}$,  and hence $\mathrm{Cay}(S_n,S)\cong 6K_4$. This completes the proof.
\end{proof}

\subsection{Normal Cayley graphs on $S_n$ with strictly second largest eigenvalue at most $1$}

A \emph{cocktail party graph} is a graph obtained from a complete graph with even order by deleting a perfect matching. In 2010, Stani\'{c} \cite{Stanic2010} obtained the following complete classification of connected regular graphs with the second largest eigenvalue at most $1$.

\begin{thm}[cf.~{\cite[Theorem 3.1]{Stanic2010}}]\label{second-ev-1}
	Every connected regular graph with the second largest eigenvalue at most $1$ is the complement of a (not necessarily connected) regular graph, where each connected component is either a connected regular line graph, a cocktail party graph, or one of the 187 exceptional connected regular graphs with order from 8 to 28 and degree from 3 to 16 described in \cite[pp.\,213--227]{CRS2004}.
\end{thm}

Using Theorems~\ref{second-ev-1} and \ref{class-linegraph}, we obtain the following classification of normal Cayley graphs on $S_n$ with strictly second largest eigenvalue at most $1$.

\begin{thm}\label{sn-cay-sec-1}
	Let $n\ge 5$, and let $S$ be a nonempty normal set in $S_n$. Then the strictly second largest eigenvalue of $\mathrm{Cay}(S_n,S)$ is at most $1$ if and only if one of the following holds:
	\begin{enumerate}[\rm (a)]
		\item $S=S_n\setminus\{e\}$, and hence $\mathrm{Cay}(S_n,S)\cong K_{n!}$;
		\item $S=S_n\setminus \, A_n$, and hence $\mathrm{Cay}(S_n,S)\cong K_{n!/2,n!/2}$;
		\item $S=A_n\setminus\{e\}$, and hence $\mathrm{Cay}(S_n,S)\cong 2K_{n!/2}$.
	\end{enumerate}
\end{thm}

\begin{proof}
Let $n\ge 5$ and $S$ be a nonempty normal set in $S_n$. If $S=S_n\setminus\{e\}, S_n\setminus A_n$ or $A_n\setminus\{e\}$, then $\mathrm{Cay}(S_n,S)\cong K_{n!}, K_{n!/2,n!/2}$ or $2K_{n!/2}$, respectively. Note that $K_{n!}$ has distinct eigenvalues $n!-1$ and $-1$, $K_{n!/2,n!/2}$ has distinct eigenvalues  $n!/2$, $0$ and $-n!/2$, and $2K_{n!/2}$ has distinct eigenvalues $n!/2-1$ and $-1$. In all three cases, the strictly second largest eigenvalue of $\mathrm{Cay}(S_n,S)$ is at most $1$. This proves the sufficiency. 

We now prove the necessity. Suppose that the strictly second largest eigenvalue of $\mathrm{Cay}(S_n,S)$ is at most $1$. 

\smallskip
\textsf{Case 1.} $\mathrm{Cay}(S_n,S)$ is connected.
\smallskip

In this case, we have $\langle S \rangle = S_n$. Since $S$ is a normal set in $S_n$, so is $S^c=(S_n\setminus\{e\})\setminus S$, and hence $\langle S^c\rangle$ is a normal subgroup of $S_n$. As $n \ge 5$, we have $\langle S^c\rangle = \{e\},A_n$ or $S_n$.
If $\langle S^c\rangle=\{e\}$, then $S^c=\emptyset$, so $S=S_n\setminus\{e\}$ and  $\mathrm{Cay}(S_n,S)\cong K_{n!}$.

Assume that $\langle S^c\rangle$ is either $A_n$ or $S_n$. Since $n\ge 5$, every connected component of $\mathrm{Cay}(S_n,S^c)$ has at least $5!/2=60$ vertices. Since $\lambda_2(\mathrm{Cay}(S_n,S))\le 1$, by Theorem~\ref{second-ev-1}, each connected component of $\mathrm{Cay}(S_n,S^c)$ is either a regular line graph, a cocktail party graph, or one of the 187 exceptional connected regular graphs. But all these exceptional graphs have at most $28$ vertices, so none of them can be a connected component of $\mathrm{Cay}(S_n,S^c)$.

Suppose that some connected component of $\mathrm{Cay}(S_n,S^c)$ is a cocktail party graph. Then the complement of the subgraph of $\mathrm{Cay}(S_n,S^c)$ induced by $\langle S^c\rangle$ is $1$-regular. Hence $S^c$ omits exactly one nonidentity element of $\langle S^c\rangle$. Since $S^c$ is a union of conjugacy classes of $S_n$, this would give a conjugacy class of size $1$ in $S_n$ other than $\{e\}$, contradicting the fact that the center of $S_n$ is trivial for $n\ge 5$. Therefore, no connected component of $\mathrm{Cay}(S_n,S^c)$ is a cocktail party graph.

Thus every connected component of $\mathrm{Cay}(S_n,S^c)$ is a regular line graph. Since a disjoint union of line graphs is again a line graph, $\mathrm{Cay}(S_n,S^c)$ is itself a line graph. It follows from Theorem~\ref{class-linegraph} that either $S^c=S_n\setminus\{e\}$ or $S^c=A_n\setminus\{e\}$. The former is impossible since it would imply $S=\emptyset$. Thus, the only possibility is that $S^c=A_n\setminus\{e\}$. Hence $S=S_n\setminus A_n$ and therefore $\mathrm{Cay}(S_n,S)\cong K_{n!/2,n!/2}$.

\smallskip
\textsf{Case 2.} $\mathrm{Cay}(S_n,S)$ is disconnected.
\smallskip

In this case, we have $\langle S\rangle=A_n$. Set
\(
S^*=(A_n\setminus\{e\})\setminus S.
\)
Then $\mathrm{Cay}(A_n,S^*)$ is the complement of $\mathrm{Cay}(A_n,S)$. If $S^*=\emptyset$, then $S=A_n\setminus\{e\}$, and hence $\mathrm{Cay}(S_n,S)\cong 2K_{n!/2}$.

Assume that $S^*\neq\emptyset$. Since $S$ is a normal set in $S_n$, $S^*$ is a union of conjugacy classes of both $S_n$ and $A_n$. Since $A_n$ is simple for $n\ge 5$, the nonempty normal set $S^*$ generates $A_n$, and hence $\mathrm{Cay}(A_n,S^*)$ is connected.
As $\lambda_2(\mathrm{Cay}(A_n,S))\le 1$, Theorem~\ref{second-ev-1} implies that $\mathrm{Cay}(A_n,S^*)$ is either a regular line graph, a cocktail party graph, or one of the 187 exceptional connected regular graphs. The same arguments as in Case 1 show that it is neither a cocktail party graph nor an exceptional graph. Hence $\mathrm{Cay}(A_n,S^*)$ must be a line graph. Therefore, $\mathrm{Cay}(S_n,S^*)$, which is the disjoint union of two copies of $\mathrm{Cay}(A_n,S^*)$, is also a line graph.
Moreover, $S^*$ is a union of conjugacy classes of $S_n$. Therefore, by Theorem~\ref{class-linegraph}, we have either $S^*=S_n\setminus\{e\}$ or $S^*=A_n\setminus\{e\}$. Since $S^*\subseteq A_n$, it follows that $S^*=A_n\setminus\{e\}$. This forces $S=\emptyset$, contradicting the fact that $\langle S\rangle=A_n$.
\end{proof}

\section{Proofs of Theorem~\ref{thm:main2} and Corollary \ref{cor:main2}}
\label{supp-large}

With the preparation in the previous section, we can finally prove Theorem~\ref{thm:main2} and Corollary \ref{cor:main2}.

\begin{proof}[Proof of Theorem~\ref{thm:main2}]
Let $n\ge 5$, and let $S$ be a nonempty normal set in $S_n$. Let $A_S, B_S$ and $C_S$ be as in \eqref{eq:ABCS}. Then $S=A_S \cup B_S \cup C_S$. Since $\mathrm{Cay}(S_n, S)$ is normal, by \eqref{eq:eigen-chara} and Table~\ref{tab:tab1}, its eigenvalue  corresponding to the partition $(n-1,1)$ is
\begin{align}
\lambda_{(n-1,1)}
& = \sum_{\sigma \in S} \tilde{\chi}_{(n-1,1)}(\sigma) \nonumber \\ 
& = \frac{1}{n-1} \sum_{\sigma \in S} (|\fix(\sigma)|-1) \nonumber \\ 
& = \frac{1}{n-1} \left(\sum_{\sigma \in A_S} (|\fix(\sigma)|-1) - \sum_{\sigma \in C_S} 1\right) \nonumber \\ 
& = \frac{1}{n-1} \left(\left(\sum_{\sigma \in A_S} |\fix(\sigma)|\right) - (|A_S|+|C_S|)\right). \label{eq:abc}
\end{align}

By our assumption \eqref{eq:ABC} and equation \eqref{eq:abc}, we have $\lambda_{(n-1,1)}\le 1$.

\smallskip
\textsf{Necessity.}
Suppose that $\mathrm{Cay}(S_n,S)$ has the Aldous property. Then its strictly second largest eigenvalue, which is attained by the standard representation $\rho_{(n-1,1)}$, is $\lambda_{(n-1,1)}\le 1$.
By Theorem~\ref{sn-cay-sec-1}, we have $S \in \{S_n\setminus\{e\}, S_n\setminus A_n, A_n\setminus\{e\}\}$.

\smallskip
\textsf{Sufficiency.}
If $S=S_n\setminus\{e\}$, then $\mathrm{Cay}(S_n,S)\cong K_{n!}$ has spectrum
\(
\{(n!-1)^1,(-1)^{n!-1}\}.
\)
The eigenvalue $n!-1$ is attained by $\rho_{(n)}$, while the strictly second largest eigenvalue $-1$ is attained by $\rho_{(n-1,1)}$.

If $S=S_n\setminus A_n$, then $\mathrm{Cay}(S_n,S)\cong K_{n!/2,n!/2}$ has spectrum
\(
\left\{\left(\frac{n!}{2}\right)^1,0^{n!-2},\left(-\frac{n!}{2}\right)^1\right\}.
\)
The eigenvalue $\frac{n!}{2}$ is attained by $\rho_{(n)}$, the eigenvalue $-\frac{n!}{2}$ is attained by $\rho_{(1^n)}$, and the strictly second largest eigenvalue $0$ is attained by $\rho_{(n-1,1)}$.

If $S=A_n\setminus\{e\}$, then $\mathrm{Cay}(S_n,S)\cong 2K_{n!/2}$ has spectrum
\(
\left\{\left(\frac{n!}{2}-1\right)^2,(-1)^{n!-2}\right\}.
\)
The eigenvalue $\frac{n!}{2}-1$ is attained by both $\rho_{(n)}$ and $\rho_{(1^n)}$, while the strictly second largest eigenvalue $-1$ is attained by $\rho_{(n-1,1)}$.

Therefore, in all three cases, $\mathrm{Cay}(S_n,S)$ has the Aldous property. 
Moreover, since in each case the strictly second largest eigenvalue is $\lambda_{(n-1,1)}\leq 1$, it follows from \eqref{eq:abc} that all three sets $S$ satisfy \eqref{eq:ABC}.

This proves the sufficiency and completes the proof.
\end{proof}

\begin{proof}[Proof of Corollary \ref{cor:main2}]
Let $n\ge 5$, and let $S$ be a nonempty normal set in $S_n$.
	
(a) Suppose that $S \subseteq \mathrm{Supp}_n(\{n-1,n\})$. Then $A_S = \emptyset$, so $S$ satisfies condition \eqref{eq:ABC}. Hence Theorem~\ref{thm:main2} applies to $\mathrm{Cay}(S_n,S)$. Since $n\ge 5$, each of $S_n\setminus\{e\}, S_n\setminus A_n$ and $A_n\setminus\{e\}$ contains an element whose support size is neither $n-1$ nor $n$. Thus, $S \notin \{S_n\setminus\{e\}, S_n\setminus A_n, A_n\setminus\{e\}\}$ and therefore, by Theorem \ref{thm:main2}, $\mathrm{Cay}(S_n,S)$ does not have the Aldous property.  

(b) Suppose that $|\overline{A}_S|-|\overline{C}_S|+2(n-1) \ge 0$ and $S \notin \{S_n\setminus\{e\}, S_n\setminus A_n, A_n\setminus\{e\}\}$. By \eqref{eq:ABCS} and \eqref{eq:barABCS}, $\{A_S, \overline{A}_S, B_S, \overline{B}_S, C_S, \overline{C}_S\}$ is a partition of $S_{n} \setminus \{e\}$. Hence
\begin{equation}
\label{sum-set-sn}
|A_S|+|\overline{A}_S|+|B_S|+|\overline{B}_S|+|C_S|+|\overline{C}_S| = n!-1. 
\end{equation}
Set 
$$
\mathcal{X} = \{ (\sigma, i) \in S_n \times [n] : \sigma(i) = i \}.
$$
We use double counting to enumerate $\mathcal{X}$. First, for each fixed $i$, there are exactly $(n-1)!$ permutations $\sigma \in S_n$ satisfying $\sigma(i) = i$. Therefore,
\begin{equation}
\label{eq:fac}
|\mathcal{X}| = \sum_{i=1}^n |\{ \sigma \in S_n : \sigma(i) = i \}| = \sum_{i=1}^n (n-1)! = n!.
\end{equation}
Second, each permutation $\sigma$ contributes exactly $|\fix(\sigma)|$ to the cardinality of $\mathcal{X}$. Thus,
\begin{align*}
	|\mathcal{X}| 
	&= 
	\sum_{\sigma \in S_n} |\fix(\sigma)|  \\
	&= 
	\sum_{k=0}^n \sum_{\sigma \in \mathrm{Supp}_n(k)} |\fix(\sigma)|   \\
	&=
	\sum_{\sigma \in \mathrm{Supp}_n(0)} |\fix(\sigma)|
	+
	\sum_{\sigma \in \mathrm{Supp}_n(1)} |\fix(\sigma)|
	+
	\sum_{k=2}^{n-1} \sum_{\sigma \in \mathrm{Supp}_n(k)} |\fix(\sigma)|
	+
	\sum_{\sigma \in \mathrm{Supp}_n(n)} |\fix(\sigma)|  \\
	&=
	n+0+
	\sum_{k=2}^{n-1} \sum_{\sigma \in \mathrm{Supp}_n(k)} |\fix(\sigma)|
	+0  \\
	&=
	n+\sum_{\sigma \in \mathrm{Supp}_n(\{2, 3, \ldots, n-2\})} |\fix(\sigma)|+\sum_{\sigma \in \mathrm{Supp}_n(n-1)} |\fix(\sigma)|.  
\end{align*}
This together with \eqref{sum-set-sn} and \eqref{eq:fac} yields
\begin{align}
\sum_{\sigma \in A_S} |\fix(\sigma)| + \sum_{\sigma \in \overline{A}_S} |\fix(\sigma)|
&= \sum_{\sigma \in \mathrm{Supp}_n(\{2, 3, \ldots, n-2\})} |\fix(\sigma)| \nonumber \\
&= n! - n - \sum_{\sigma \in \mathrm{Supp}_n(n-1)} |\fix(\sigma)| \label{eq:fac1} \\
&= n! - n - \sum_{\sigma \in B_S} |\fix(\sigma)| - \sum_{\sigma \in \overline{B}_S} |\fix(\sigma)| \nonumber \\ 
&= n! - n - |B_S| - |\overline{B}_S| \nonumber \\ 
&= |A_S| + |\overline{A}_S| + |C_S| + |\overline{C}_S| + 1 - n. \nonumber 
\end{align}
Since $|\overline{A}_S|-|\overline{C}_S|+2(n-1) \ge 0$ by our assumption, it follows that
\begin{align*}
|A_S| + |C_S| + n - 1 - \sum_{\sigma \in A_S} |\fix(\sigma)| 
&= \sum_{\sigma \in \overline{A}_S} |\fix(\sigma)|-|\overline{A}_S|-|\overline{C}_S|+2(n-1) \\
&= \sum_{\sigma \in \overline{A}_S} (|\fix(\sigma)|-1)-|\overline{C}_S|+2(n-1) \\
& \geq |\overline{A}_S|-|\overline{C}_S|+2(n-1) \\
& \geq 0.
\end{align*}
That is, $S$ satisfies condition \eqref{eq:ABC}. Since $S \notin \{S_n\setminus\{e\}, S_n\setminus A_n, A_n\setminus\{e\}\}$ by our assumption, it follows from Theorem \ref{thm:main2} that $\mathrm{Cay}(S_n,S)$ does not have the Aldous property.

(c) Suppose that $\mathrm{Supp}_n(n)\subseteq S$ and $S \neq S_n\setminus\{e\}$. Since $\mathrm{Supp}_n(n)$ contains both even and odd permutations, $S$ is neither $S_n\setminus A_n$ nor $A_n\setminus\{e\}$. 
Moreover, $C_S = \mathrm{Supp}_n(n)$ and hence $\overline{C}_S = \emptyset$. Thus, $|\overline{A}_S|-|\overline{C}_S|+2(n-1) = |\overline{A}_S|+2(n-1) \ge 0$ and therefore by (b), $\mathrm{Cay}(S_n,S)$ does not have the Aldous property. 
\end{proof}

\medskip
\noindent \textbf{Acknowledgements}~~Chenhui Lv was supported by the Outstanding Doctoral Students Overseas Study Program of the University of Science and Technology of China. Sanming Zhou was supported by a Discovery Project (DP250104965) of the Australian Research Council.

\end{document}